\documentclass[12pt]{amsart}
\usepackage[T1]{fontenc}
\usepackage[utf8]{inputenc}
\usepackage{lmodern}
\usepackage{amsmath,amssymb,amsthm,mathrsfs}
\usepackage{microtype}
\usepackage[left=1.05in,right=1.05in,top=1.0in,bottom=1.03in]{geometry}
\usepackage{hyperref}
\hypersetup{colorlinks=true,linkcolor=black,citecolor=black,urlcolor=black}
\newtheorem{theorem}{Theorem}[section]
\newtheorem{proposition}[theorem]{Proposition}
\newtheorem{lemma}[theorem]{Lemma}
\newtheorem{corollary}[theorem]{Corollary}
\theoremstyle{definition}

\theoremstyle{remark}
\newtheorem{remark}[theorem]{Remark}

\newcommand{\D}{\mathbb D}
\newcommand{\B}{\mathbb B}
\newcommand{\C}{\mathbb C}

\newcommand{\OO}{\mathcal O}

\newcommand{\ord}{\operatorname{ord}}
\newcommand{\Ree}{\operatorname{Re}}
\newcommand{\arctanh}{\operatorname{arctanh}}
\newcommand{\arsinh}{\operatorname{arsinh}}

\title[Invariant metrics on a real-analytic pseudoconvex domain]{Non-comparable invariant metrics and failure of Gromov hyperbolicity\\on a real-analytic pseudoconvex domain}
\author{Armen Edigarian}
\address{Faculty of Mathematics and Computer Science, Jagiellonian University, Prof. St. \L ojasiewicza 6, 30-348 Krak\'ow, Poland}
\email{armen.edigarian@uj.edu.pl}

\begin{document}

\begin{abstract}
We study the explicit bounded pseudoconvex domain
\[
 \Omega=\{(z_1,z_2,w)\in\C^3:
 |w|^2+|z_1z_2(z_1-z_2)|^2+|z_1|^8+|z_2|^8<1\}.
\]
Its boundary is real analytic and has finite D'Angelo type everywhere.  We prove two assertions about the invariant metrics of this domain.  First, $(\Omega,k_\Omega)$ is not Gromov hyperbolic; in fact arbitrarily fat Kobayashi geodesic triangles occur in every neighborhood of the type-$8$ point $(0,0,1)$.  Second, the Carath\'eodory--Reiffen and Kobayashi--Royden metrics are not uniformly comparable.  Along an explicit family of tangential directions the Carath\'eodory metric tends to zero while the Kobayashi metric has a positive limit.  For the corresponding distances there are points $x_j,y_j\to(0,0,1)$ such that $c_\Omega(x_j,y_j)\to0$ while $k_\Omega(x_j,y_j)\to\infty$.  This gives, among bounded real-analytic pseudoconvex domains, an affirmative answer to Problem~7.3.2 of the Several Complex Variables Problem List.
\end{abstract}

\maketitle

\section{Introduction}\label{sec:intro}

Consider the domain
\begin{equation}\label{eq:domain}
 \Omega=
 \left\{(z_1,z_2,w)\in\C^3:
 |w|^2+|z_1z_2(z_1-z_2)|^2+|z_1|^8+|z_2|^8<1
 \right\}.
\end{equation}
There are two reasons to look at this very simple domain.  Both come from the same cubic fiber, but they answer two rather different questions.

The first question concerns Gromov hyperbolicity.  If a bounded pseudoconvex domain is smoothly bounded and of finite D'Angelo type, must its Kobayashi metric be Gromov hyperbolic?  The known positive results make this a natural guess.  Balogh and Bonk proved the strongly pseudoconvex case \cite{BaloghBonk}.  Zimmer proved that for smooth bounded convex domains finite type is exactly the condition for Gromov hyperbolicity, and obtained positive results for locally convexifiable finite-type domains \cite{Zimmer}.  Fiacchi proved the finite-type statement in $\C^2$ \cite{Fiacchi}, Zhang proved it in arbitrary dimension under the Levi-corank-one hypothesis \cite{Zhang}, and a recent preprint of Sarkar proves the corresponding equivalence for smooth bounded $\C$-convex domains \cite{Sarkar}.  The general higher-dimensional pseudoconvex case is different.

\begin{theorem}\label{thm:main}
For the domain \eqref{eq:domain}:
\begin{enumerate}
\item $\Omega$ is bounded and pseudoconvex, and $\partial\Omega$ is real analytic;
\item every boundary point has D'Angelo one-type at most $8$;
\item the point
\[
 p=(0,0,1)
\]
has D'Angelo one-type exactly $8$;
\item $(\Omega,k_\Omega)$ is not Gromov hyperbolic.
\end{enumerate}
More strongly, for every neighborhood $V$ of $p$ and every $M>0$, there is a Kobayashi geodesic triangle contained in $V$ which has a point on one side at Kobayashi distance greater than $M$ from the union of the other two sides.
\end{theorem}

The second question concerns the two basic holomorphically contractible metrics.  Let $\gamma_\Omega$ denote the Carath\'eodory--Reiffen metric and $\kappa_\Omega$ the Kobayashi--Royden metric.  One always has
\[
 \gamma_\Omega\le \kappa_\Omega,
 \qquad
 c_\Omega\le k_\Omega.
\]
Thus uniform comparability means a reverse inequality with one constant on the whole domain.  Krantz asked whether there is a smoothly bounded pseudoconvex domain on which the Kobayashi and Carath\'eodory metrics are not comparable; this is Problem~7.3.2 in \cite{SCVProblemList}.  Earlier examples concerning non-comparability with the Bergman metric were studied by Diederich--Forn\ae ss--Herbort \cite{DiederichFornaessHerbort}.  Strongly pseudoconvex domains behave differently: the two infinitesimal metrics have the same leading boundary behavior \cite{Graham}, and the distances are asymptotically equivalent \cite{Venturini}.  On smooth finite-type domains in $\C^2$, Catlin's estimates make the invariant metrics uniformly comparable \cite{CatlinMetrics}; related positive results hold under additional higher-dimensional geometric hypotheses, for example Levi corank one \cite{BalakumarMahajanVerma}.

The same domain \eqref{eq:domain} gives the opposite behavior in $\C^3$.

\begin{theorem}\label{thm:noncomparable}
On $\Omega$, the Carath\'eodory--Reiffen and Kobayashi--Royden metrics are not uniformly comparable.  More precisely, there are points $x_j\to p$ and nonzero vectors $X_j\in\C^3$ such that
\begin{equation}\label{eq:intro-ratio}
 \frac{\gamma_\Omega(x_j;X_j)}{\kappa_\Omega(x_j;X_j)}\longrightarrow0.
\end{equation}
The distances are not uniformly comparable either.  More strongly, there are pairs $x_j,y_j\to p$ such that
\begin{equation}\label{eq:intro-distance-strong}
 c_\Omega(x_j,y_j)\longrightarrow0,
 \qquad
 k_\Omega(x_j,y_j)\longrightarrow\infty.
\end{equation}
In particular,
\begin{equation}\label{eq:intro-distance-ratio}
 \frac{c_\Omega(x_j,y_j)}{k_\Omega(x_j,y_j)}\longrightarrow0.
\end{equation}
\end{theorem}

Thus the example answers the question in \cite[Problem~7.3.2]{SCVProblemList} already among bounded real-analytic pseudoconvex domains.  The same domain also has finite D'Angelo type everywhere, with type exactly $8$ at $p$.  The failure is local at that point.

Let us explain why the same construction produces both phenomena.  Put
\begin{equation}\label{eq:P}
 P(z_1,z_2)=z_1z_2(z_1-z_2).
\end{equation}
Then
\begin{equation}\label{eq:r}
 r(z_1,z_2,w)=|w|^2+|P(z_1,z_2)|^2+|z_1|^8+|z_2|^8-1
\end{equation}
is a global real-analytic plurisubharmonic defining function.  On the interior fibers
\[
 P(z_1,z_2)=a^3
\]
write $(z_1,z_2)=a(u,v)$.  The fiber becomes the fixed affine cubic
\begin{equation}\label{eq:cubic}
 uv(u-v)=1.
\end{equation}
The ratio $u/v$ maps this cubic onto
\[
 S=\C\setminus\{0,1\}.
\]
The cubic therefore contains a hyperbolic direction of arbitrarily large diameter.  Independently, the $w$-variable gives a disc direction.  The three orders
\[
 a^6,\qquad a^7,\qquad a^8
\]
are the relevant ones: $|P|^2$ contributes $a^6$, we keep an inward normal margin of order $a^7$, and the eighth-order terms contribute only $a^8$ on fixed compact subsets of the cubic.  Thus increasingly large parts of the two independent directions are represented inside $\Omega$ as $a\to0$.

For the Kobayashi metric, the main point is to show that passing to the ambient domain does not shorten either direction.  The $w$-direction has an explicit global holomorphic left inverse.  On
\[
 \Omega\setminus\{P=0\}
\]
the ratio
\[
 \pi(z_1,z_2,w)=\frac{z_1}{z_2}
\]
is a holomorphic map into $S=\C\setminus\{0,1\}$.  The set $\{P=0\}\cap\Omega$ tends to infinite Kobayashi distance from the points under consideration.  Royden localization then gives the required lower bound in the cubic direction.  The two lower bounds together give the four-point configurations used in the proof of Theorem~\ref{thm:main}.

For the Carath\'eodory metric the same cubic has a different property.  Its projective completion is a smooth cubic obtained by adding three points.  Every bounded holomorphic function on the affine cubic therefore extends to the compact completion and is constant.  Since larger and larger subsets of the affine cubic map into $\Omega$, bounded holomorphic functions on $\Omega$ become asymptotically constant in this direction.  This gives the vanishing of the Carath\'eodory quantities, whereas the Kobayashi lower bound remains positive by localization.  These are \eqref{eq:intro-ratio} and \eqref{eq:intro-distance-strong}.

It is also useful to see where the standard positive theories stop.  At $p$ the Levi corank is $2$.  The Catlin and D'Angelo multitypes are respectively $(1,6,6)$ and $(1,6,8)$, so $p$ is neither locally convexifiable nor $h$-extendible.  The domain is not $\C$-convex either.  We check all of this directly in Section~\ref{sec:finite-type}.

Section~\ref{sec:finite-type} deals with the boundary geometry.  Section~\ref{sec:localization} gives the distance and infinitesimal localization estimates.  Section~\ref{sec:cubic} constructs the holomorphic product maps.  Section~\ref{sec:lowerbounds} proves lower bounds in the two Kobayashi directions.  Section~\ref{sec:caratheodory} proves Theorem~\ref{thm:noncomparable}.  Sections~\ref{sec:fourpoint} and \ref{sec:triangles} prove the Gromov-hyperbolicity statements.

\section{Preliminaries}\label{sec:prelim}

We write
\[
 \D=\{\zeta\in\C:|\zeta|<1\},
 \qquad
 \D_\rho=\{\zeta\in\C:|\zeta|<\rho\}.
\]
Our normalization of the Poincar\'e distance is
\begin{equation}\label{eq:poincare}
 k_\D(0,t)=\arctanh t=\frac12\log\frac{1+t}{1-t},
 \qquad 0\le t<1.
\end{equation}
Thus
\begin{equation}\label{eq:poincare-pm}
 k_\D(-t,t)=2\arctanh t.
\end{equation}
For a domain $D\subset\C^n$, $\kappa_D$ and $k_D$ denote the Kobayashi--Royden metric and Kobayashi distance.  The Carath\'eodory--Reiffen metric and Carath\'eodory distance are denoted by $\gamma_D$ and $c_D$.  With the normalization \eqref{eq:poincare},
\begin{equation}\label{eq:caratheodory-metric}
 \gamma_D(z;X)
 =\sup_{f\in\OO(D,\D)}
 \frac{|df_zX|}{1-|f(z)|^2}.
\end{equation}
The Carath\'eodory distance is
\begin{equation}\label{eq:caratheodory-distance}
 c_D(x,y)
 =\sup_{f\in\OO(D,\D)} k_\D(f(x),f(y)).
\end{equation}
We always have
\begin{equation}\label{eq:basic-comparison}
 \gamma_D\le\kappa_D,
 \qquad
 c_D\le k_D.
\end{equation}
We repeatedly use holomorphic contraction:
\begin{equation}\label{eq:contraction}
 k_{D_2}(F(x),F(y))\le k_{D_1}(x,y)
\end{equation}
for every holomorphic map $F:D_1\to D_2$, and the analogous infinitesimal inequalities for $\kappa$ and $\gamma$.

We shall use the four-point formulation of Gromov hyperbolicity.  A metric space $(X,d)$ is Gromov hyperbolic if there is $\delta\ge0$ such that for every four points $x_1,x_2,x_3,x_4$, if
\[
 \Sigma_1=d(x_1,x_2)+d(x_3,x_4),\quad
 \Sigma_2=d(x_1,x_3)+d(x_2,x_4),\quad
 \Sigma_3=d(x_1,x_4)+d(x_2,x_3),
\]
then the two largest of $\Sigma_1,\Sigma_2,\Sigma_3$ differ by at most $2\delta$.  This is equivalent, for geodesic metric spaces, to uniform thinness of geodesic triangles; see \cite{BridsonHaefliger}.

For a smooth real hypersurface $M=\{r=0\}$ and $q\in M$, the D'Angelo one-type is
\begin{equation}\label{eq:dangelo}
 \Delta_1(M,q)
 =
 \sup_\gamma
 \frac{\ord_0(r\circ\gamma)}{\ord_0(\gamma-q)},
\end{equation}
where $\gamma:(\C,0)\to(\C^n,q)$ ranges over nonconstant holomorphic curve germs.  This number is independent of the choice of smooth defining function; see \cite{DAngelo}.

\section{The domain is smooth, pseudoconvex, and of type at most eight}\label{sec:finite-type}

Set
\[
 P(z_1,z_2)=z_1z_2(z_1-z_2)
\]
and let $\Omega$ be given by \eqref{eq:domain}.  Define
\begin{equation}\label{eq:G}
 G:\C^3\longrightarrow\C^4,
 \qquad
 G(z_1,z_2,w)=\bigl(w,P(z_1,z_2),z_1^4,z_2^4\bigr).
\end{equation}
Then
\begin{equation}\label{eq:rG}
 r=\|G\|^2-1,
 \qquad
 \Omega=G^{-1}(\B^4).
\end{equation}

\begin{proposition}\label{prop:basic-domain}
The domain $\Omega$ is bounded, connected and pseudoconvex, and its boundary is a smooth real-analytic hypersurface.
\end{proposition}

\begin{proof}
Boundedness is immediate from
\[
 |w|^2+|z_1|^8+|z_2|^8<1.
\]
The defining function $r$ is real analytic and plurisubharmonic, because it is a sum of squared moduli of holomorphic polynomials.  Since $r<0$ on $\Omega$ and $r\to0$ at the boundary, $-\log(-r)$ is a plurisubharmonic exhaustion of $\Omega$.  Thus $\Omega$ is pseudoconvex.

To verify smoothness and connectedness simultaneously, observe that for $0\le t\le1$,
\begin{equation}\label{eq:radial}
 r(tz_1,tz_2,tw)+1
 =t^2|w|^2+t^6|P(z_1,z_2)|^2+t^8\bigl(|z_1|^8+|z_2|^8\bigr).
\end{equation}
Thus $\Omega$ is star-shaped with respect to the origin and is therefore connected.  If $(z_1,z_2,w)\in\partial\Omega$, differentiation at $t=1$ gives
\begin{equation}\label{eq:radial-derivative}
 \left.\frac{d}{dt}r(tz_1,tz_2,tw)\right|_{t=1}
 =2|w|^2+6|P|^2+8\bigl(|z_1|^8+|z_2|^8\bigr)>0.
\end{equation}
Hence $dr\ne0$ on $\partial\Omega$.
\end{proof}

The finite-type estimate is particularly transparent in the representation \eqref{eq:rG}.

\begin{proposition}\label{prop:type}
For every $q\in\partial\Omega$,
\[
 \Delta_1(\partial\Omega,q)\le8.
\]
At
\[
 p=(0,0,1)
\]
one has
\[
 \Delta_1(\partial\Omega,p)=8.
\]
\end{proposition}

\begin{proof}
Fix $q\in\partial\Omega$ and a nonconstant holomorphic curve germ
\[
 \gamma:(\C,0)\longrightarrow(\C^3,q).
\]
Write
\[
 m=\ord_0(\gamma-q)
\]
and
\[
 \ell=\ord_0(G\circ\gamma-G(q)),
\]
where the order of a vector-valued holomorphic germ is the minimum of the orders of its components.

We claim that
\begin{equation}\label{eq:ell-4m}
 \ell\le4m.
\end{equation}
Indeed, choose a coordinate component of $\gamma-q$ having order $m$.  If that component is the $w$-component, then the first component of $G\circ\gamma-G(q)$ has order $m$.  If it is the $z_j$-component and $q_j\ne0$, then
\[
 z_j(\zeta)^4-q_j^4
 =4q_j^3(z_j(\zeta)-q_j)+O\bigl((z_j(\zeta)-q_j)^2\bigr),
\]
so its order is $m$.  If $q_j=0$, the corresponding fourth-power component has order $4m$.  This proves \eqref{eq:ell-4m}.

Put
\[
 H(\zeta)=G\circ\gamma(\zeta)-G(q)
 =A_\ell\zeta^\ell+O(\zeta^{\ell+1}),
 \qquad A_\ell\ne0.
\]
Since $\|G(q)\|=1$,
\begin{equation}\label{eq:r-curve}
 r\circ\gamma
 =2\Ree\langle H,G(q)\rangle+\|H\|^2.
\end{equation}
The first term is harmonic.  In the second term the coefficient of
$\zeta^\ell\bar\zeta^\ell$ is $\|A_\ell\|^2>0$.  A harmonic function has no nonzero mixed coefficient, so this coefficient cannot be cancelled by the first term.  Consequently
\[
 \ord_0(r\circ\gamma)\le2\ell\le8m.
\]
Taking the supremum over $\gamma$ gives the upper bound $\Delta_1\le8$.

At $p=(0,0,1)$ consider
\[
 \gamma(\zeta)=(\zeta,0,1).
\]
Then
\[
 r\circ\gamma(\zeta)=|\zeta|^8.
\]
Thus the type at $p$ is at least $8$, and the preceding upper bound gives equality.
\end{proof}

\begin{remark}[The boundary geometry at $p$]\label{rem:geometry-at-p}
At $p=(0,0,1)$ the complex tangent space is the $(z_1,z_2)$-plane, and the Levi form of $r$ vanishes identically on it.  Thus the Levi corank is $2$, so the Levi-corank-one theorem of Zhang \cite{Zhang} does not apply.

We shall also compare the D'Angelo and Catlin multitypes.  Write $w=1-\sigma$.  Near $p$,
\begin{equation}\label{eq:normal-form-p}
 r=-2\Ree\sigma+|\sigma|^2+|P(z_1,z_2)|^2+|z_1|^8+|z_2|^8.
\end{equation}
The D'Angelo multitype at $p$ is
\[
 (1,6,8).
\]
The last entry is Proposition~\ref{prop:type}.  To compute the $2$-type, first note that every complex two-plane through the origin in the coordinates $(z_1,z_2,\sigma)$ contains a nonzero complex tangential direction.  Along such a direction the order of contact is at least $6$.  Hence the one-type of every two-plane section is at least $6$.

Conversely, take
\[
 L_\lambda=\{z_2=\lambda z_1\},
 \qquad \lambda\notin\{0,1\}.
\]
Writing $z_1=\xi$, the restriction of \eqref{eq:normal-form-p} to $L_\lambda$ is
\[
 -2\Ree\sigma+|\sigma|^2
 +|\lambda(1-\lambda)|^2|\xi|^6
 +(1+|\lambda|^8)|\xi|^8.
\]
Its one-type is exactly $6$: the curve $\sigma=0$ has contact $6$, and the same mixed-coefficient argument as in Proposition~\ref{prop:type} gives the reverse inequality for every holomorphic curve in this two-plane.  Thus the D'Angelo $2$-type is $6$.

Catlin's multitype is
\[
 (1,6,6).
\]
The weight $(1,\frac16,\frac16)$ is distinguished in the coordinates of \eqref{eq:normal-form-p}.  It remains to show that no distinguished weight preceding it in Catlin's lexicographic ordering can occur after a general adapted holomorphic change of coordinates.  Let $(\omega,y_1,y_2)$ be adapted coordinates.  On the plane $\omega=0$, write the inverse change of coordinates in the form
\[
 z=Ay+O(|y|^2),
 \qquad
 \sigma=h(y),
 \qquad
 A\in\mathrm{GL}_2(\C).
\]
The homogeneous term of degree three in $P(z(y))$ is $P(Ay)$.  If $|\alpha|=3$, then the coefficient of $y^\alpha\overline{y}^\alpha$ in the restriction of $r$ to $\omega=0$ is
\[
 |[y^\alpha]h|^2+|[y^\alpha]P(Ay)|^2.
\]
It is therefore positive whenever $[y^\alpha]P(Ay)\ne0$.  The harmonic term $-2\Ree h$ cannot cancel such a mixed coefficient, and the terms $|z_1|^8+|z_2|^8$ do not contribute in bidegree $(3,3)$.

Let the two tangential weights be ordered as $\lambda_1\ge\lambda_2$.  If $\lambda_1<1/6$, then every cubic monomial in $P(Ay)$ has holomorphic weighted degree less than $1/2$, and the preceding positive mixed coefficient has weighted degree less than $1$, which is impossible for a distinguished weight.  If $\lambda_1=1/6$ and $\lambda_2<1/6$, the same argument shows that every cubic monomial involving $y_2$ must vanish.  Hence one would have
\[
 P(Ay)=c\,y_1^3.
\]
This is impossible because $P$ is the product of three distinct linear factors, a property preserved by an invertible linear change.  Thus the Catlin multitype is $(1,6,6)$; see \cite{Catlin} for the definition and \cite{YuMultitypes} for the convex case.

For a finite-type convex hypersurface the Catlin and D'Angelo multitypes agree \cite{YuMultitypes}.  Since both multitypes are local biholomorphic invariants, $p$ is not locally convexifiable.  Their inequality also shows that $p$ is not $h$-extendible; see \cite{YuWeighted}.  The corresponding weighted model
\[
 \{(z_1,z_2,\sigma):2\Ree\sigma>|P(z_1,z_2)|^2\}
\]
contains complex affine lines, for example $\{z_1=0,\ \sigma=1\}$.
\end{remark}

\begin{remark}[The domain is not $\C$-convex]\label{rem:not-c-convex}
The $\C$-convex results do not cover the example either.  Take
\[
 0<c<\frac14,
 \qquad
 c^8<\varepsilon<\frac{c^6}{16},
 \qquad
 |w_0|^2=1-c^8-\varepsilon.
\]
On the complex line $\{z_1=c,\ w=w_0\}$ the intersection with $\Omega$ is
\begin{equation}\label{eq:c-convex-slice}
 c^2|z_2(c-z_2)|^2+|z_2|^8<\varepsilon.
\end{equation}
It contains $z_2=0$ and $z_2=c$.  On the other hand \eqref{eq:c-convex-slice} implies
\[
 |z_2(c-z_2)|<\frac{\sqrt\varepsilon}{c}<\frac{c^2}{4}.
\]
On the vertical line $z_2=c/2+iy$ one has
\[
 |z_2(c-z_2)|=\frac{c^2}{4}+y^2\ge\frac{c^2}{4}.
\]
Hence the section in \eqref{eq:c-convex-slice} does not meet the line $\Ree z_2=c/2$, while it contains the points $z_2=0$ and $z_2=c$ on opposite sides of that line.  The section is therefore disconnected, and $\Omega$ is not $\C$-convex.
\end{remark}

We shall later use actual geodesics, so we first note that the Kobayashi metric is proper.

\begin{lemma}\label{lem:proper}
The metric space $(\Omega,k_\Omega)$ is a proper geodesic metric space.
\end{lemma}

\begin{proof}
The holomorphic map $G$ in \eqref{eq:G} maps $\Omega$ properly into $\B^4$.  Indeed, if $K\Subset\B^4$, then $\|G\|\le1-\eta$ on $G^{-1}(K)$ for some $\eta>0$, so $G^{-1}(K)$ is a closed subset of a bounded sublevel set of $r$ lying a positive Euclidean distance from $\partial\Omega$.

Fix $x_0\in\Omega$ and $R>0$.  By holomorphic contraction,
\[
 k_{\B^4}(G(x_0),G(x))\le k_\Omega(x_0,x).
\]
Hence the closed $k_\Omega$-ball of radius $R$ around $x_0$ is contained in the inverse image under $G$ of the closed $k_{\B^4}$-ball of radius $R$ around $G(x_0)$.  The latter is compact in $\B^4$, so properness of $G$ implies compactness of the $k_\Omega$-ball.

The Kobayashi distance is the integrated distance of the Kobayashi--Royden metric.  Thus $(\Omega,k_\Omega)$ is a proper length space, and every proper length space is geodesic; see \cite[Chapter~I.3]{BridsonHaefliger}.
\end{proof}

\section{Localization away from a distant analytic set}\label{sec:localization}

The ratio $z_1/z_2$ will be used as a holomorphic map into $S=\C\setminus\{0,1\}$.  Such an $S$-valued map is available only after removing the set $\{P=0\}$.  The following form of Royden localization says that this causes arbitrarily little loss when the removed set is far away in the Kobayashi distance \cite{Royden}.

\begin{lemma}[Distance localization]\label{lem:localization}
Let $D$ be a Kobayashi-hyperbolic domain, let $U\subset D$ be open, and let
\[
 \pi:U\longrightarrow Y
\]
be holomorphic into a Kobayashi-hyperbolic complex manifold $Y$.  Suppose $x,y\in U$ and choose $M$ such that
\[
 k_D(x,y)<M<d_0,
 \qquad
 d_0:=k_D(x,D\setminus U).
\]
Then
\begin{equation}\label{eq:localization}
 k_Y(\pi(x),\pi(y))
 \le
 \coth(d_0-M)\,k_D(x,y),
\end{equation}
or equivalently
\begin{equation}\label{eq:localization-lower}
 k_D(x,y)
 \ge
 \tanh(d_0-M)\,k_Y(\pi(x),\pi(y)).
\end{equation}
\end{lemma}

\begin{proof}
By Royden's theorem, $k_D$ is the integrated distance of $\kappa_D$.  Choose a piecewise $C^1$ curve $\gamma$ from $x$ to $y$ with $\kappa_D$-length less than $M$.  If $q\in\gamma$, then
\[
 k_D(q,D\setminus U)
 \ge d_0-k_D(x,q)>d_0-M.
\]
Let $f:\D\to D$ be a holomorphic disc with $f(0)=q$.  If
\[
 0<\rho<\tanh(d_0-M),
\]
then
\begin{equation}\label{eq:disc-in-U}
 f(\D_\rho)\subset U.
\end{equation}
Indeed, if $f(\zeta)\notin U$ for some $|\zeta|<\rho$, contraction would give
\[
 d_0-M
 <k_D(q,D\setminus U)
 \le k_D(q,f(\zeta))
 \le k_\D(0,\zeta)
 <d_0-M,
\]
a contradiction.

Restrict $f$ to $\D_\rho$ and rescale.  This gives
\[
 \kappa_U(q;X)\le\rho^{-1}\kappa_D(q;X).
\]
Since $\pi$ is holomorphic,
\[
 \kappa_Y(\pi(q);d\pi_qX)
 \le\rho^{-1}\kappa_D(q;X).
\]
Integrating along $\gamma$,
\[
 k_Y(\pi(x),\pi(y))
 \le\rho^{-1}\ell_{\kappa_D}(\gamma).
\]
Now let $\ell_{\kappa_D}(\gamma)\downarrow k_D(x,y)$ and then
$\rho\uparrow\tanh(d_0-M)$.
\end{proof}

The same disc restriction gives the infinitesimal form that we shall use later.

\begin{lemma}[Infinitesimal localization]\label{lem:inf-localization}
Let $D,U,Y$ and $\pi:U\to Y$ be as in Lemma~\ref{lem:localization}.  If $x\in U$ and
\[
 d_0=k_D(x,D\setminus U)>0,
\]
then for every $X\in T_xD$,
\begin{equation}\label{eq:inf-localization}
 \kappa_D(x;X)
 \ge
 \tanh(d_0)\,
 \kappa_Y(\pi(x);d\pi_xX).
\end{equation}
\end{lemma}

\begin{proof}
Fix $0<\rho<\tanh d_0$.  If $f:\D\to D$ is a holomorphic disc with $f(0)=x$, then $f(\D_\rho)\subset U$; otherwise contraction would give a point of $D\setminus U$ at Kobayashi distance smaller than $d_0$ from $x$.  Rescaling the disc gives
\[
 \kappa_U(x;X)\le\rho^{-1}\kappa_D(x;X).
\]
Since $\pi$ is holomorphic,
\[
 \kappa_Y(\pi(x);d\pi_xX)\le\kappa_U(x;X).
\]
Thus
\[
 \kappa_D(x;X)\ge\rho\,\kappa_Y(\pi(x);d\pi_xX).
\]
Let $\rho\uparrow\tanh d_0$.
\end{proof}

\section{The hyperbolic cubic and the product maps}\label{sec:cubic}

Let
\begin{equation}\label{eq:S}
 S=\C\setminus\{0,1\}
\end{equation}
and fix once and for all a universal covering
\begin{equation}\label{eq:tau}
 \tau:\D\longrightarrow S.
\end{equation}
We will use only two facts about $\tau$: it is holomorphic and onto, and as a covering of hyperbolic Riemann surfaces it is a local isometry for the Poincar\'e metrics.

Since $\tau(\tau-1)$ never vanishes on the simply connected disc, choose a holomorphic cube root
\begin{equation}\label{eq:v}
 v^3=\frac1{\tau(\tau-1)}
\end{equation}
and put
\begin{equation}\label{eq:u}
 u=\tau v.
\end{equation}
Then
\begin{equation}\label{eq:cubic-param}
 uv(u-v)=1,
 \qquad
 \frac uv=\tau.
\end{equation}
Thus $(u,v)$ parametrizes the affine cubic \eqref{eq:cubic} on its universal cover, and the ratio coordinate is exactly the hyperbolic surface $S$.

For $a>0$ small define
\begin{equation}\label{eq:parameters}
 c_a=a^3,
 \qquad
 s_a=\sqrt{1-a^6},
 \qquad
 b_a=\sqrt{1-\frac{a^7}{1-a^6}},
\end{equation}
and
\begin{equation}\label{eq:Wa}
 W_a(\eta)
 =s_a\frac{b_a+\eta}{1+b_a\eta}.
\end{equation}
Since $s_a^2(1-b_a^2)=a^7$, the disc-automorphism identity gives
\begin{equation}\label{eq:Wa-identity}
 1-a^6-|W_a(\eta)|^2
 =a^7\frac{1-|\eta|^2}{|1+b_a\eta|^2}.
\end{equation}
In particular,
\begin{equation}\label{eq:wa0}
 W_a(0)=\sqrt{1-a^6-a^7}.
\end{equation}

Define
\begin{equation}\label{eq:Phi}
 \Phi_a(\zeta,\eta)
 =\bigl(au(\zeta),av(\zeta),W_a(\eta)\bigr).
\end{equation}
By \eqref{eq:cubic-param},
\[
 P(au(\zeta),av(\zeta))=a^3,
\]
so
\begin{equation}\label{eq:rPhi}
 r(\Phi_a(\zeta,\eta))
 =a^8\bigl(|u(\zeta)|^8+|v(\zeta)|^8\bigr)
 -a^7\frac{1-|\eta|^2}{|1+b_a\eta|^2}.
\end{equation}

\begin{lemma}[Product maps on fixed bidiscs]\label{lem:product-map}
Fix $0<\rho<1$.  For all sufficiently small $a>0$,
\begin{equation}\label{eq:product-inclusion}
 \Phi_a(\D_\rho\times\D_\rho)\subset\Omega.
\end{equation}
Consequently, for $\zeta,\zeta',\eta,\eta'\in\D_\rho$,
\begin{equation}\label{eq:product-upper}
 k_\Omega\bigl(\Phi_a(\zeta,\eta),\Phi_a(\zeta',\eta')\bigr)
 \le
 \max\{k_{\D_\rho}(\zeta,\zeta'),k_{\D_\rho}(\eta,\eta')\}.
\end{equation}
\end{lemma}

\begin{proof}
On $|\eta|\le\rho$,
\[
 \frac{1-|\eta|^2}{|1+b_a\eta|^2}
 \ge\frac{1-\rho}{1+\rho}.
\]
Also
\[
 M_\rho:=\max_{|\zeta|\le\rho}
 \bigl(|u(\zeta)|^8+|v(\zeta)|^8\bigr)<\infty.
\]
Thus \eqref{eq:rPhi} is negative as soon as
\[
 aM_\rho<\frac{1-\rho}{1+\rho}.
\]
This proves the inclusion.  The distance estimate is holomorphic contraction together with
\[
 k_{\D_\rho\times\D_\rho}
 =\max\{k_{\D_\rho},k_{\D_\rho}\}.
\]
\end{proof}

The useful form of the upper bound is obtained by letting the parameter radius tend to one.

\begin{corollary}\label{cor:asymptotic-upper}
For fixed $x=(\zeta,\eta)$ and $y=(\zeta',\eta')$ in $\D^2$,
\begin{equation}\label{eq:asymptotic-upper}
 \limsup_{a\to0}
 k_\Omega(\Phi_a(x),\Phi_a(y))
 \le
 k_{\D^2}(x,y)
 =\max\{k_\D(\zeta,\zeta'),k_\D(\eta,\eta')\}.
\end{equation}
\end{corollary}

\begin{proof}
Choose $\rho<1$ containing the four fixed parameter values.  Lemma~\ref{lem:product-map} gives the same estimate with $\D_\rho$ in place of $\D$.  Let $a\to0$ and then $\rho\uparrow1$.
\end{proof}

The order of quantifiers is the following: the parameter points are fixed first, and then $a$ is taken small.  No fixed $a$ is supposed to carry the whole bidisc.

\section{Lower bounds in the two directions}\label{sec:lowerbounds}

The product maps give the upper bounds.  We now prove lower bounds in each of the two directions.

\subsection{The \texorpdfstring{$w$}{w}-factor has an exact left inverse}

Define
\begin{equation}\label{eq:F}
 F:\Omega\longrightarrow\B^2,
 \qquad
 F(z_1,z_2,w)=\bigl(P(z_1,z_2),w\bigr).
\end{equation}
For $c=c_a=a^3$ and $s=s_a=\sqrt{1-c^2}$ put
\begin{equation}\label{eq:La}
 L_a(q,w)=\frac{sw}{1-cq},
 \qquad (q,w)\in\B^2.
\end{equation}
A direct calculation gives
\begin{equation}\label{eq:Lball}
 |1-cq|^2-(1-c^2)|w|^2
 =|q-c|^2+(1-c^2)(1-|q|^2-|w|^2)>0,
\end{equation}
so $L_a:\B^2\to\D$ is holomorphic.  Now set
\begin{equation}\label{eq:nua}
 \nu_a
 =\frac{L_a\circ F-b_a}{1-b_aL_a\circ F}:
 \Omega\longrightarrow\D.
\end{equation}
On the image of $\Phi_a$ one has $P=a^3=c$, and therefore
\[
 L_a(F(\Phi_a(\zeta,\eta)))
 =\frac{W_a(\eta)}{s_a}
 =\frac{b_a+\eta}{1+b_a\eta}.
\]
Hence
\begin{equation}\label{eq:leftinverse}
 \nu_a\circ\Phi_a(\zeta,\eta)=\eta.
\end{equation}
Thus, whenever the points belong to the image of the product map,
\begin{equation}\label{eq:w-lower}
 k_\Omega\bigl(\Phi_a(\zeta,\eta),\Phi_a(\zeta',\eta')\bigr)
 \ge k_\D(\eta,\eta').
\end{equation}
This gives the required lower bound in the $w$-direction.

\subsection{The divisor \texorpdfstring{$\{P=0\}$}{P=0} goes to infinity}

Set
\begin{equation}\label{eq:Z}
 Z=\Omega\cap\{P=0\},
 \qquad
 U=\Omega\setminus Z.
\end{equation}
On $U$ the ratio
\begin{equation}\label{eq:pi}
 \pi(z_1,z_2,w)=\frac{z_1}{z_2}
\end{equation}
is holomorphic and takes values in $S=\C\setminus\{0,1\}$.  By \eqref{eq:cubic-param},
\begin{equation}\label{eq:piPhi}
 \pi\circ\Phi_a(\zeta,\eta)=\tau(\zeta).
\end{equation}
The removal of $Z$ is necessary because on $Z$ the ratio either is undefined or takes one of the omitted values $0,1$.  We now estimate the Kobayashi distance to $Z$.

\begin{lemma}\label{lem:divisor-distance}
Whenever $\Phi_a(\zeta,0)\in\Omega$,
\begin{equation}\label{eq:divisor-distance}
 k_\Omega(\Phi_a(\zeta,0),Z)
 \ge
 \arctanh\frac1{\sqrt{1+a}}
 =\arsinh(a^{-1/2}).
\end{equation}
The bound is independent of $\zeta$.
\end{lemma}

\begin{proof}
Put
\[
 W=W_a(0)=\sqrt{1-a^6-a^7}.
\]
Since $P(au(\zeta),av(\zeta))=a^3$,
\[
 F(\Phi_a(\zeta,0))=(a^3,W)
\]
for every $\zeta$.  Define
\begin{equation}\label{eq:HW}
 H_W(q,w)=\frac{\sqrt{1-W^2}\,q}{1-Ww},
 \qquad (q,w)\in\B^2.
\end{equation}
The same identity as in \eqref{eq:Lball} shows that $H_W:\B^2\to\D$ is holomorphic.  If $z\in Z$, then $H_W(F(z))=0$.  At the point $\Phi_a(\zeta,0)$,
\[
 1-W^2=a^6+a^7=a^6(1+a),
\]
and hence
\[
 H_W(F(\Phi_a(\zeta,0)))
 =\frac{\sqrt{a^6+a^7}\,a^3}{a^6+a^7}
 =\frac1{\sqrt{1+a}}.
\]
Holomorphic contraction therefore gives
\[
 k_\Omega(\Phi_a(\zeta,0),z)
 \ge
 \arctanh\frac1{\sqrt{1+a}}.
\]
Take the infimum over $z\in Z$.  Finally,
\[
 \arctanh\frac1{\sqrt{1+a}}=\arsinh(a^{-1/2}).
\]
\end{proof}

Thus the set removed in order to obtain the $S$-valued ratio map tends to infinite Kobayashi distance.  The localization lemma now gives the second lower bound.

\begin{proposition}[Lower bound in the cubic direction]\label{prop:ratio-lower}
For fixed $\zeta,\zeta'\in\D$,
\begin{equation}\label{eq:ratio-lower}
 \liminf_{a\to0}
 k_\Omega\bigl(\Phi_a(\zeta,0),\Phi_a(\zeta',0)\bigr)
 \ge
 k_S(\tau(\zeta),\tau(\zeta')).
\end{equation}
\end{proposition}

\begin{proof}
Choose
\[
 M>k_\D(\zeta,\zeta').
\]
By Corollary~\ref{cor:asymptotic-upper}, for all small $a$,
\[
 k_\Omega(\Phi_a(\zeta,0),\Phi_a(\zeta',0))<M.
\]
On the other hand Lemma~\ref{lem:divisor-distance} gives
\[
 d_a:=k_\Omega(\Phi_a(\zeta,0),Z)
 \ge\arsinh(a^{-1/2})\longrightarrow\infty.
\]
Apply Lemma~\ref{lem:localization} to
\[
 D=\Omega,
 \qquad U=\Omega\setminus Z,
 \qquad \pi=z_1/z_2.
\]
Using \eqref{eq:piPhi},
\[
 k_\Omega\bigl(\Phi_a(\zeta,0),\Phi_a(\zeta',0)\bigr)
 \ge
 \tanh(d_a-M)\,
 k_S(\tau(\zeta),\tau(\zeta')).
\]
Let $a\to0$.
\end{proof}

\section{The Carath\'eodory metric in the cubic direction}\label{sec:caratheodory}

The preceding proposition gives a nonzero Kobayashi lower bound in the cubic direction.  The Carath\'eodory metric behaves differently: bounded holomorphic functions become asymptotically constant on larger and larger subsets of the cubic.  We first prove this statement.

Let
\[
 \Gamma=\{(u,v)\in\C^2:uv(u-v)=1\}.
\]
Its projective closure is
\begin{equation}\label{eq:projective-cubic}
 \overline\Gamma=\{[U:V:W]\in\mathbb P^2:UV(U-V)=W^3\}.
\end{equation}
This is a smooth cubic.  Indeed, in the affine coordinates $W=1$ a singular point would have to satisfy
\[
 V(2U-V)=0,
 \qquad
 U(U-2V)=0,
 \qquad
 UV(U-V)=1,
\]
which is impossible.  At infinity there are exactly three points,
\[
 [0:1:0],\qquad[1:0:0],\qquad[1:1:0],
\]
and the gradient of $UV(U-V)-W^3$ is nonzero at each of them.  Thus $\Gamma$ is obtained from the compact Riemann surface $\overline{\Gamma}$ by removing these three points.  Every bounded holomorphic function on $\Gamma$ extends across the punctures and is therefore constant.

Define
\begin{equation}\label{eq:iota}
 \iota_a:\Gamma\longrightarrow\C^3,
 \qquad
 \iota_a(u,v)=\bigl(au,av,W_a(0)\bigr).
\end{equation}
On $\Gamma$ we have $P(au,av)=a^3$, and therefore
\begin{equation}\label{eq:r-iota}
 r\circ\iota_a
 =a^8(|u|^8+|v|^8)-a^7.
\end{equation}
Thus
\begin{equation}\label{eq:Gamma-a}
 \Gamma_a
 =\{(u,v)\in\Gamma:|u|^8+|v|^8<a^{-1}\}
\end{equation}
embeds holomorphically into $\Omega$, and $\Gamma_a\nearrow\Gamma$ as $a\to0$.

\begin{lemma}[Asymptotic constancy on the cubic]\label{lem:bounded-cubic}
Fix $q\in\Gamma$ and a compact set $K\Subset\Gamma$.  Then
\begin{equation}\label{eq:bounded-cubic}
 \sup\left\{
 \sup_{x\in K}|f(\iota_a(x))|:
 f\in\OO(\Omega,\D),\ f(\iota_a(q))=0
 \right\}\longrightarrow0
\end{equation}
as $a\to0$.
\end{lemma}

\begin{proof}
Suppose not.  Then for some $\varepsilon>0$ there are $a_j\to0$, maps $f_j\in\OO(\Omega,\D)$ with $f_j(\iota_{a_j}(q))=0$, and points $x_j\in K$ such that
\[
 |f_j(\iota_{a_j}(x_j))|\ge\varepsilon.
\]
After passing to a subsequence, $x_j\to x\in K$.  The functions
\[
 g_j=f_j\circ\iota_{a_j}
\]
are defined on an exhaustion of $\Gamma$ and satisfy $|g_j|\le1$.  A diagonal Montel argument gives a subsequence converging locally uniformly on $\Gamma$ to a bounded holomorphic function $g$.  Since $g(q)=0$ and every bounded holomorphic function on $\Gamma$ is constant, $g\equiv0$.  Local uniform convergence near $x$ contradicts $|g_j(x_j)|\ge\varepsilon$.
\end{proof}

The corresponding statement for the Carath\'eodory distance follows at once.

\begin{proposition}[Carath\'eodory distance along the cubic]\label{prop:caratheodory-distance}
For fixed $q,q'\in\Gamma$,
\begin{equation}\label{eq:caratheodory-general-zero}
 c_\Omega(\iota_a(q),\iota_a(q'))\longrightarrow0.
\end{equation}
In particular, for fixed $\zeta,\zeta'\in\D$,
\begin{equation}\label{eq:caratheodory-zero}
 c_\Omega\bigl(\Phi_a(\zeta,0),\Phi_a(\zeta',0)\bigr)\longrightarrow0.
\end{equation}
\end{proposition}

\begin{proof}
Given $f\in\OO(\Omega,\D)$, compose it with the automorphism of $\D$ sending $f(\iota_a(q))$ to $0$.  Lemma~\ref{lem:bounded-cubic}, with $K=\{q,q'\}$, shows uniformly in $f$ that the normalized value at $q'$ tends to zero.  By the definition of the Carath\'eodory distance this is exactly \eqref{eq:caratheodory-general-zero}.  For \eqref{eq:caratheodory-zero}, take
\[
 q=(u(\zeta),v(\zeta)),
 \qquad
 q'=(u(\zeta'),v(\zeta')).
\]
\end{proof}

We now pass to infinitesimal metrics.  Fix $\zeta\in\D$ and set
\begin{equation}\label{eq:x-X-a}
 x_a=\Phi_a(\zeta,0),
 \qquad
 X_a=\partial_\zeta\Phi_a(\zeta,0)
 =a\bigl(u'(\zeta),v'(\zeta),0\bigr).
\end{equation}
The vector $X_a$ is nonzero: otherwise $\tau'=(u/v)'$ would vanish at $\zeta$, which is impossible for a covering map.

\begin{proposition}[Infinitesimal non-comparability]\label{prop:inf-separation}
For every fixed $\zeta\in\D$,
\begin{equation}\label{eq:gamma-zero}
 \gamma_\Omega(x_a;X_a)\longrightarrow0,
\end{equation}
and
\begin{equation}\label{eq:kappa-limit}
 \kappa_\Omega(x_a;X_a)
 \longrightarrow
 \kappa_\D(\zeta;1)>0.
\end{equation}
Consequently
\begin{equation}\label{eq:metric-ratio-zero}
 \frac{\gamma_\Omega(x_a;X_a)}{\kappa_\Omega(x_a;X_a)}\longrightarrow0.
\end{equation}
\end{proposition}

\begin{proof}
Put
\[
 q=(u(\zeta),v(\zeta))\in\Gamma,
 \qquad
 \xi=(u'(\zeta),v'(\zeta))\in T_q\Gamma.
\]
Then $x_a=\iota_a(q)$ and $X_a=d(\iota_a)_q\xi$.

For the Carath\'eodory metric, normalize $f\in\OO(\Omega,\D)$ by a disc automorphism so that $f(x_a)=0$.  Lemma~\ref{lem:bounded-cubic} applied on a fixed compact coordinate disc in $\Gamma$ around $q$ says that $f\circ\iota_a\to0$ there uniformly, uniformly over all such normalized $f$.  Cauchy's estimate therefore gives
\[
 \sup_{f(x_a)=0}|d(f\circ\iota_a)_q\xi|\longrightarrow0.
\]
By \eqref{eq:caratheodory-metric}, this is \eqref{eq:gamma-zero}.

For the upper bound in \eqref{eq:kappa-limit}, fix $\rho<1$ with $|\zeta|<\rho$.  For small $a$, Lemma~\ref{lem:product-map} gives a holomorphic map
\[
 \Phi_a:\D_\rho^2\longrightarrow\Omega.
\]
Infinitesimal contraction at $(\zeta,0)$ in the vector $(1,0)$ yields
\[
 \kappa_\Omega(x_a;X_a)
 \le \kappa_{\D_\rho}(\zeta;1).
\]
First let $a\to0$, then $\rho\uparrow1$, to obtain
\begin{equation}\label{eq:kappa-limsup}
 \limsup_{a\to0}\kappa_\Omega(x_a;X_a)
 \le\kappa_\D(\zeta;1).
\end{equation}

For the lower bound, use $U=\Omega\setminus Z$ and $\pi=z_1/z_2$.  Lemma~\ref{lem:divisor-distance} gives
\[
 d_a:=k_\Omega(x_a,Z)\longrightarrow\infty.
\]
By Lemma~\ref{lem:inf-localization},
\[
 \kappa_\Omega(x_a;X_a)
 \ge
 \tanh(d_a)\,
 \kappa_S\bigl(\pi(x_a);d\pi_{x_a}X_a\bigr).
\]
Using $\pi\circ\Phi_a(\zeta,0)=\tau(\zeta)$,
\[
 d\pi_{x_a}X_a=\tau'(\zeta).
\]
The covering $\tau:\D\to S$ is a local Poincar\'e isometry, hence
\[
 \kappa_S(\tau(\zeta);\tau'(\zeta))
 =\kappa_\D(\zeta;1).
\]
Letting $a\to0$ gives the reverse liminf inequality.  Together with \eqref{eq:kappa-limsup}, this proves \eqref{eq:kappa-limit}.
\end{proof}

\begin{proof}[Proof of Theorem~\ref{thm:noncomparable}]
The infinitesimal assertion follows from Proposition~\ref{prop:inf-separation}; since $x_a\to p$, the failure occurs arbitrarily close to $p$.  If one wants Euclidean unit vectors, replace $X_a$ by $X_a/\|X_a\|$; the ratio in \eqref{eq:metric-ratio-zero} is unchanged by homogeneity.

For the distances, first choose fixed $\zeta,\zeta'\in\D$ with $\tau(\zeta)\ne\tau(\zeta')$.  Proposition~\ref{prop:ratio-lower} and Proposition~\ref{prop:caratheodory-distance} give
\[
 \liminf_{a\to0}
 k_\Omega\bigl(\Phi_a(\zeta,0),\Phi_a(\zeta',0)\bigr)
 \ge
 k_S(\tau(\zeta),\tau(\zeta'))>0,
\]
while the corresponding Carath\'eodory distance tends to zero.  Hence no uniform constant $C$ can satisfy $k_\Omega\le Cc_\Omega$.

We can obtain the stronger statement \eqref{eq:intro-distance-strong} by a diagonal choice.  For every integer $j\ge1$, choose $\zeta_j,\zeta'_j\in\D$ such that
\[
 k_S(\tau(\zeta_j),\tau(\zeta'_j))>2j;
\]
this is possible because $\tau$ is onto and $S$ has infinite Poincar\'e diameter.  Once this pair is fixed, Propositions~\ref{prop:ratio-lower} and \ref{prop:caratheodory-distance} allow us to choose $a_j>0$ so small that
\[
 k_\Omega\bigl(\Phi_{a_j}(\zeta_j,0),\Phi_{a_j}(\zeta'_j,0)\bigr)>j,
 \qquad
 c_\Omega\bigl(\Phi_{a_j}(\zeta_j,0),\Phi_{a_j}(\zeta'_j,0)\bigr)<\frac1j,
\]
and, at the same time, both points are at Euclidean distance less than $1/j$ from $p$.  With
\[
 x_j=\Phi_{a_j}(\zeta_j,0),
 \qquad
 y_j=\Phi_{a_j}(\zeta'_j,0),
\]
we obtain \eqref{eq:intro-distance-strong}, and hence also \eqref{eq:intro-distance-ratio}.
\end{proof}

The two arguments use the same cubic in different ways.  Global bounded holomorphic functions become asymptotically constant on it, while the $S$-valued ratio on $\Omega\setminus Z$ gives a Kobayashi lower bound because $Z$ is far away in the Kobayashi distance.

\section{The four-point contradiction}\label{sec:fourpoint}

We now prove non-hyperbolicity.

Choose arbitrary $\zeta_1,\zeta_2\in\D$.  Let $\zeta_0$ be the midpoint of a Poincar\'e geodesic joining them and put
\begin{equation}\label{eq:R-midpoint}
 R=\frac12k_\D(\zeta_1,\zeta_2),
 \qquad
 t=\tanh R.
\end{equation}
Thus
\[
 k_\D(\zeta_1,\zeta_0)
 =k_\D(\zeta_0,\zeta_2)
 =k_\D(0,t)=R.
\]
For small $a$ set
\begin{equation}\label{eq:fourpoints}
 \begin{aligned}
 A_a&=\Phi_a(\zeta_1,0),&\qquad B_a&=\Phi_a(\zeta_2,0),\\
 C_a&=\Phi_a(\zeta_0,-t),& D_a&=\Phi_a(\zeta_0,t).
 \end{aligned}
\end{equation}
Write
\begin{equation}\label{eq:N}
 N=k_S(\tau(\zeta_1),\tau(\zeta_2)).
\end{equation}
Proposition~\ref{prop:ratio-lower} gives
\begin{equation}\label{eq:AB-liminf}
 \liminf_{a\to0}k_\Omega(A_a,B_a)\ge N.
\end{equation}
The exact left inverse \eqref{eq:w-lower} gives
\begin{equation}\label{eq:CD-lower}
 k_\Omega(C_a,D_a)\ge k_\D(-t,t)=2R.
\end{equation}
Finally Corollary~\ref{cor:asymptotic-upper} gives, for every cross pair,
\begin{equation}\label{eq:cross-limsup}
 \limsup_{a\to0}k_\Omega(A_a,C_a),\quad
 \limsup_{a\to0}k_\Omega(A_a,D_a),\quad
 \limsup_{a\to0}k_\Omega(B_a,C_a),\quad
 \limsup_{a\to0}k_\Omega(B_a,D_a)
 \le R.
\end{equation}
Thus the two distinguished pairs $(A_a,B_a)$ and $(C_a,D_a)$ have lower bounds $N$ and $2R$, respectively, while each of the four cross-distances has upper limit at most $R$.

\begin{proof}[Proof of Theorem~\ref{thm:main}, non-hyperbolicity]
Assume that $(\Omega,k_\Omega)$ is $\delta$-hyperbolic in the four-point sense.  Applied to $A_a,B_a,C_a,D_a$, the four-point inequality gives
\begin{align*}
 k_\Omega(A_a,B_a)+k_\Omega(C_a,D_a)
 \le{}&
 \max\{k_\Omega(A_a,C_a)+k_\Omega(B_a,D_a),\\
 &\hspace{2.4cm}k_\Omega(A_a,D_a)+k_\Omega(B_a,C_a)\}+2\delta.
\end{align*}
Let $a\to0$.  From \eqref{eq:AB-liminf}--\eqref{eq:cross-limsup},
\[
 N+2R\le2R+2\delta,
\]
so
\begin{equation}\label{eq:S-bounded}
 k_S(\tau(\zeta_1),\tau(\zeta_2))\le2\delta
\end{equation}
for every $\zeta_1,\zeta_2\in\D$.

But $\tau$ is onto, while $S=\C\setminus\{0,1\}$ has infinite Poincar\'e diameter.  Indeed its Poincar\'e metric is complete, and a complete finite-dimensional Riemannian manifold of finite diameter is compact by Hopf--Rinow, whereas $S$ is not compact.  Thus \eqref{eq:S-bounded} is impossible.  Hence $(\Omega,k_\Omega)$ is not Gromov hyperbolic.
\end{proof}

For fixed $\zeta_1,\zeta_2$, all four points in \eqref{eq:fourpoints} tend to $p=(0,0,1)$ as $a\to0$: the first two coordinates are $O(a)$ and the third tends to $1$.  The failure of the four-point inequality therefore already occurs arbitrarily close to the single boundary point $p$.

\section{Arbitrarily fat geodesic triangles near \texorpdfstring{$p$}{p}}\label{sec:triangles}

The four-point argument is enough for non-hyperbolicity.  To prove the stronger local statement in Theorem~\ref{thm:main}, we use one elementary metric observation.

\begin{lemma}\label{lem:metric-fat}
Let $(X,d)$ be a geodesic metric space.  Let $A,B,C,D\in X$, let $E$ be the midpoint of a geodesic segment $[A,B]$, and set
\[
 \alpha=d(E,[A,C]\cup[B,C]),
 \qquad
 \beta=d(E,[A,D]\cup[B,D]).
\]
If
\[
 m=\max\{d(A,C),d(A,D),d(B,C),d(B,D)\},
\]
then
\begin{equation}\label{eq:metric-fat}
 \alpha+\beta
 \ge
 \frac12\bigl(d(A,B)+d(C,D)\bigr)-m.
\end{equation}
\end{lemma}

\begin{proof}
Take $V\in[A,C]$.  Since the segment is geodesic,
\[
 d(E,C)
 \le d(E,V)+d(V,C)
 =d(E,V)+d(A,C)-d(A,V).
\]
Also
\[
 d(A,V)\ge d(A,E)-d(E,V)=\frac12d(A,B)-d(E,V).
\]
Hence
\[
 d(E,C)\le2d(E,V)+m-\frac12d(A,B).
\]
The same estimate holds for $V\in[B,C]$, so
\[
 d(E,C)\le2\alpha+m-\frac12d(A,B).
\]
Similarly,
\[
 d(E,D)\le2\beta+m-\frac12d(A,B).
\]
Now use $d(C,D)\le d(C,E)+d(E,D)$ and rearrange.
\end{proof}

Return to the four points \eqref{eq:fourpoints}.  Let $E_a$ be the midpoint of a Kobayashi geodesic $[A_a,B_a]$ and put
\begin{align*}
 \alpha_a&=k_\Omega(E_a,[A_a,C_a]\cup[B_a,C_a]),\\
 \beta_a&=k_\Omega(E_a,[A_a,D_a]\cup[B_a,D_a]).
\end{align*}
By Lemma~\ref{lem:metric-fat}, \eqref{eq:AB-liminf}, \eqref{eq:CD-lower}, and \eqref{eq:cross-limsup},
\begin{equation}\label{eq:fat-liminf}
 \liminf_{a\to0}(\alpha_a+\beta_a)\ge\frac N2.
\end{equation}
Thus if $N>4M$, then for all sufficiently small $a$ one of the two triangles
\[
 (A_a,B_a,C_a),
 \qquad
 (A_a,B_a,D_a)
\]
has a point $E_a$ on the common side whose distance from the other two sides is greater than $M$.

It remains to show that, for fixed parameters, the whole triangles eventually lie in every prescribed neighborhood of $p$.  This follows from the normal coordinate.

\begin{lemma}[Kobayashi balls approach $p$]\label{lem:ball-near-p}
Fix $\zeta\in\D$ and $M>0$, and put $A_a=\Phi_a(\zeta,0)$.  If $X_a\in\Omega$ satisfies
\[
 k_\Omega(A_a,X_a)\le M,
\]
then
\[
 X_a\longrightarrow p=(0,0,1)
\]
uniformly over all such $X_a$ as $a\to0$.
\end{lemma}

\begin{proof}
The projection
\[
 \operatorname{pr}_w:\Omega\to\D,
 \qquad
 (z_1,z_2,w)\mapsto w,
\]
is holomorphic.  The $w$-coordinate of $A_a$ is
\[
 w_a=W_a(0)=\sqrt{1-a^6-a^7}\longrightarrow1,
\]
independently of $\zeta$.  Hence
\[
 k_\D(w_a,w(X_a))\le M.
\]
Put $q=\tanh M$.  The pseudohyperbolic ball of radius $q$ centered at the real point $w_a$ gives
\begin{equation}\label{eq:w-near-one}
 |1-w(X_a)|
 \le\frac{1+q}{1-q}(1-w_a)
 =e^{2M}(1-w_a).
\end{equation}
Thus $w(X_a)\to1$ uniformly.  Since $X_a\in\Omega$,
\[
 |z_1(X_a)|^8+|z_2(X_a)|^8
 <1-|w(X_a)|^2
 \le2|1-w(X_a)|,
\]
so the first two coordinates tend uniformly to zero as well.
\end{proof}

\begin{proposition}[Local failure of thin triangles]\label{prop:local-fat}
Let $V$ be any neighborhood of $p$ and let $M>0$.  There is a Kobayashi geodesic triangle contained in $V$ with a point on one side at distance greater than $M$ from the union of the other two sides.
\end{proposition}

\begin{proof}
Because $\tau$ is onto and $S$ has infinite diameter, choose $\zeta_1,\zeta_2\in\D$ so that
\[
 N=k_S(\tau(\zeta_1),\tau(\zeta_2))>4M.
\]
Fix these parameters, and define $R$, $A_a,B_a,C_a,D_a$ as in Section~\ref{sec:fourpoint}.  By \eqref{eq:fat-liminf}, for all sufficiently small $a$ one of the two triangles has thickness greater than $M$ at $E_a$.

For these fixed parameters, Corollary~\ref{cor:asymptotic-upper} gives, after decreasing $a$ if necessary,
\[
 k_\Omega(A_a,B_a)\le2R+1
\]
and every cross-distance is at most $R+1$.  Hence every point of either triangle lies in the $k_\Omega$-ball of radius $3R+3$ around $A_a$.  Lemma~\ref{lem:ball-near-p} says that this whole ball is contained in $V$ for sufficiently small $a$.  The required triangle is therefore contained in $V$.
\end{proof}

Theorem~\ref{thm:main} is proved.

\medskip

{\bf AI Statements.} OpenAI ChatGPT was used in revising the exposition, checking parts of the literature, and checking and simplifying some of the arguments.  The author takes full responsibility for the mathematical content of the paper.

\end{document}